\documentclass[12pt, reqno]{amsart}
\usepackage{}
\usepackage{xcolor} 
\usepackage{cancel}
\usepackage{stmaryrd}
\usepackage{amsmath}
\usepackage{amsfonts}
\usepackage{amssymb}
\usepackage[all,cmtip]{xy}           
\usepackage{xspace}
\usepackage{bbding}
\usepackage{makecell}
\usepackage{txfonts}
\usepackage{enumerate}
\usepackage[shortlabels]{enumitem}
\usepackage{cite}

\usepackage{tikz}
\usepackage{titletoc}

\usepackage{ifpdf}
\ifpdf
 \usepackage[colorlinks,final,hyperindex]{hyperref}
\else
 \usepackage[colorlinks,final,backref=page,hyperindex,hypertex]{hyperref}
\fi
\usepackage[active]{srcltx}

\makeatletter

\begin{document}
\newcommand {\emptycomment}[1]{} 

\baselineskip=15pt
\newcommand{\nc}{\newcommand}
\newcommand{\delete}[1]{}
\nc{\mfootnote}[1]{\footnote{#1}} 
\nc{\todo}[1]{\tred{To do:} #1}

\nc{\mlabel}[1]{\label{#1}}  
\nc{\mcite}[1]{\cite{#1}}  
\nc{\mref}[1]{\ref{#1}}  
\nc{\meqref}[1]{\eqref{#1}} 
\nc{\mbibitem}[1]{\bibitem{#1}} 

\delete{
\nc{\mlabel}[1]{\label{#1}  
{\hfill \hspace{1cm}{\bf{{\ }\hfill(#1)}}}}
\nc{\mcite}[1]{\cite{#1}{{\bf{{\ }(#1)}}}}  
\nc{\mref}[1]{\ref{#1}{{\bf{{\ }(#1)}}}}  
\nc{\meqref}[1]{\eqref{#1}{{\bf{{\ }(#1)}}}} 
\nc{\mbibitem}[1]{\bibitem[\bf #1]{#1}} 
}

\newcommand {\comment}[1]{{\marginpar{*}\scriptsize\textbf{Comments:} #1}}
\nc{\mrm}[1]{{\rm #1}}
\nc{\id}{\mrm{id}}  \nc{\Id}{\mrm{Id}}
\nc{\admset}{\{\pm x\}\cup (-x+K^{\times}) \cup K^{\times} x^{-1}}

\def\a{\alpha}
\def\ad{associative D-}
\def\padm{$P$-admissible~}
\def\asi{ASI~}
\def\aybe{aYBe~}
\def\b{\beta}
\def\bd{\boxdot}
\def\bbf{\overline{f}}
\def\bF{\overline{F}}
\def\bbF{\overline{\overline{F}}}
\def\bbbf{\overline{\overline{f}}}
\def\bg{\overline{g}}
\def\bG{\overline{G}}
\def\bbG{\overline{\overline{G}}}
\def\bbg{\overline{\overline{g}}}
\def\bT{\overline{T}}
\def\bt{\overline{t}}
\def\bbT{\overline{\overline{T}}}
\def\bbt{\overline{\overline{t}}}
\def\bR{\overline{R}}
\def\br{\overline{r}}
\def\bbR{\overline{\overline{R}}}
\def\bbr{\overline{\overline{r}}}
\def\bu{\overline{u}}
\def\bU{\overline{U}}
\def\bbU{\overline{\overline{U}}}
\def\bbu{\overline{\overline{u}}}
\def\bw{\overline{w}}
\def\bW{\overline{W}}
\def\bbW{\overline{\overline{W}}}
\def\bbw{\overline{\overline{w}}}
\def\btl{\blacktriangleright}
\def\btr{\blacktriangleleft}
\def\calo{\mathcal{O}}
\def\ci{\circ}
\def\d{\delta}
\def\dd{\diamondsuit}
\def\D{\Delta}
\def\frakB{\mathfrak{B}}
\def\G{\Gamma}
\def\g{\gamma}
\def\gg{\mathfrak{g}}
\def\hh{\mathfrak{h}}
\def\k{\kappa}
\def\l{\lambda}
\def\ll{\mathfrak{L}}
\def\lh{\leftharpoonup}
\def\lr{\longrightarrow}
\def\N{Nijenhuis~}
\def\o{\otimes}
\def\om{\omega}
\def\opa{\cdot_{A}}
\def\opb{\cdot_{B}}
\def\p{\psi}
\def\r{\rho}
\def\ra{\rightarrow}
\def\rbs{Rota-Baxter system}
\def\rep{representation~}
\def\rh{\rightharpoonup}
\def\rr{{\Phi_r}}
\def\s{\sigma}
\def\srbs{symmetric Rota-Baxter system}
\def\st{\star}
\def\ss{\mathfrak{sl}_2}
\def\ti{\times}
\def\tl{\triangleright}
\def\tr{\triangleleft}
\def\v{\varepsilon}
\def\vp{\varphi}
\def\vth{\vartheta}
\def\wn{\widetilde{N}}
\def\wb{\widetilde{\beta}}

\newtheorem{thm}{Theorem}[section]
\newtheorem{lem}[thm]{Lemma}
\newtheorem{cor}[thm]{Corollary}
\newtheorem{pro}[thm]{Proposition}
\theoremstyle{definition}
\newtheorem{defi}[thm]{Definition}
\newtheorem{ex}[thm]{Example}
\newtheorem{rmk}[thm]{Remark}
\newtheorem{pdef}[thm]{Proposition-Definition}
\newtheorem{condition}[thm]{Condition}
\newtheorem{question}[thm]{Question}
\renewcommand{\labelenumi}{{\rm(\alph{enumi})}}
\renewcommand{\theenumi}{\alph{enumi}}
\newcommand{\End}{\mathrm{End}}

\nc{\ts}[1]{\textcolor{purple}{MTS:#1}}
\nc{\zc}[2]{\textcolor{blue}{ZC:#2}}
\font\cyr=wncyr10



\title{Classical Yang-Baxter equations and Nijenhuis operators for Lie algebras}

 \author{Haiying Li}
 \address{School of Mathematics and Statistics, Henan Normal University, Xinxiang 453007, China}
         \email{lihaiying@htu.edu.cn}

\author{Tianshui Ma \textsuperscript{*}}
\address{School of Mathematics and Statistics, Henan Normal University, Xinxiang 453007, China}
         \email{matianshui@htu.edu.cn}

 \thanks{\textsuperscript{*}Corresponding author}

\date{\today}

 \begin{abstract} In this paper the conditions that when a Lie algebra is Nijenhuis are investigated. On the other hand, the relations between the solutions of the classical Yang-Baxter equation and Nijenhuis operators $N$ on a Lie algebra and $P$ on a Lie coalgebra are derived by means of the bialgebraic theory for Nijenhuis Lie algebras. Furthermore all the Nijenhuis operators on $\mathfrak{sl}_2$ under the standard Cartan-Weyl basis are given.
 \end{abstract}

\subjclass[2020]{
17B38,  
17B62,   
16T25,   
17A30.  
}

\keywords{classical Yang-Baxter equation; 2-cocycle; Nijenhuis operator; Lie bialgebra}

\maketitle

\vspace{-1.2cm}

\tableofcontents

\allowdisplaybreaks
\bigskip

\section{Introduction }\label{se:ip}
 One of the aims of this paper is to prove when a Lie algebra is Nijenhuis, which means that Lie algebras and Nijenhuis operators are closely related naturally. The other aim is to find the compatible conditions between the solutions of the classical Yang-Baxter equation and Nijenhuis operators on Lie algebra and Lie coalgebra by establishing the bialgebraic theory for Nijenhuis Lie algebras.

 As is well known that a pair $(\ll, [,])$ is called a {\bf Lie algebra} if the linear map $[,]: \ll\otimes\ll\lr \ll$ satisfies
 \begin{eqnarray}
 &[x, y]=-[y, x],&\label{eq:anti}\\
 &[[x, y], z]+[[y, z], x]+[[z, x], y]=0,&\label{eq:j}
 \end{eqnarray}
 where $x, y, z \in \ll$. The most best-known quantum group is $U_q(\ss)$, which is deformation of the universal enveloping algebra of the 3-dimensional complex Lie algebra $\ss$. $\ss$ is the smallest simple Lie algebra. Concretely, the structure $[,]$ of $\ss$ with respect to the standard Cartan-Weyl basis $\{e, f, h\}$ can be given as follows:
 \smallskip
 \begin{center}
  \begin{tabular}{r|rrr}
  $[,]$ & $e$  & $f$ & $h$  \\
  \hline
  $e$ & $0$  & $h$  & $-2e$\\
  $f$ & $-h$ & $0$ & $2f$\\
  $h$ &  $2e$ & $-2f$& $0$ \\
   \end{tabular}.
  \end{center}  \smallskip

 A {\bf \N Lie algebra} \cite{NR} is a triple $((\ll, [,]), N)$ including a Lie algebra $(\ll, [,])$ and a \N operator $N$ on $\ll$, i.e., a linear map $N: \ll\lr \ll$ satisfies
 \begin{eqnarray}\label{eq:na}\label{eq:n}
 [N(x), N(y)]+N^2[x, y]=N[N(x), y]+N[x, N(y)],~\forall~x, y\in \ll.
 \end{eqnarray}
 Nijenhuis Lie algebras have applied to finite or infinite integrable systems, cohomology and deformations theory, almost complex structures,  Nijenhuis geometry, etc.

 A Lie bialgebra consists of a Lie algebra and a Lie coalgebra satisfying a compatibility condition given by a cocycle. This structure originally arose in the study of Hamiltonian mechanics and Poisson-Lie groups \cite{Dr}. Lie bialgebras can be equivalently described in terms of Manin triples of Lie algebras. The classical $r$-matrix was introduced in the classical inverse scattering method and provides a natural foundation for the Hamiltonian formalism. Semenov-Tian-Shansky first established a relation between the classical
 $r$-matrix and Rota-Baxter operators of weight 0 \cite{STS, Bai2}. Nijenhuis operators exhibit a formal similarity to Rota-Baxter operators.

 This paper unifies two central themes in Lie algebra theory-bialgebraic structures and Nijenhuis operator theory-within a coherent framework. We address the following fundamental questions: Under what conditions does a Lie (co)algebra admit a Nijenhuis structure? What are the natural compatibility conditions between a classical $r$-matrix and a Nijenhuis operator?

 The paper is organized as follows. Section \ref{se:main} explores the close relationship between Lie algebras and Nijenhuis operators. In Subsection \ref{se:qt}, we classify all triangular structures on $\ss$ (see Example \ref{ex:qt}). Subsection \ref{se:symp} presents all 2-cocycles on $\ss$ (Example \ref{ex:symp}) and proves that a dual triangular Lie bialgebra naturally induces a 2-cocycle (Proposition \ref{pro:caybetosymp}). Building on these foundations, Subsection \ref{se:n} establishes that a Lie algebra equipped with both a triangular and a dual triangular Lie bialgebra structure, together with a 2-cocycle, gives rise to Nijenhuis operators (Theorem \ref{thm:n}). We then apply this result explicitly to $\ss$ (Example \ref{ex:nsl2}).

 Section \ref{se:rep} lays the groundwork for establishing a bialgebraic structure on a Nijenhuis Lie algebra. Subsection \ref{se:drep} introduces the notion of a representation of a Nijenhuis Lie algebra (Definition \ref{de:repnlie}) and examines its dual representation (Lemma \ref{lem:dualrep}). Building on this, Subsection \ref{se:cdrep} presents the concepts of a Nijenhuis Lie coalgebra and its corepresentation, derived through compatible Lie coalgebras (Theorem \ref{thm:comcn}, Definitions \ref{de:nc} and \ref{de:ncm}). Finally, Subsection \ref{se:con} considers the dual counterpart of Theorem \ref{thm:n} (Theorem \ref{thm:con}), which provides a constructive method for obtaining Nijenhuis Lie coalgebras.

 In Section \ref{se:nliebialg}, we focus on the structure of \N Lie bialgebras. In subsection \ref{se:equi}, we introduce the notion of \N Lie bialgebras (see Definition \ref{de:nliebialg}) and also give the equivalent characterizations by using matched pairs of Nijenhuis Lie algebras (see Theorem \ref{thm:matchandnlieb}) and Manin triples of Nijenhuis Frobenius Lie algebra (see Theorem \ref{thm:matchandmanin}), which are all compatible with the case of Lie algebras. Some examples of \N Lie bialgebras on $\ss$ are presented (see Example \ref{ex:nliebialg}). In subsection \ref{se:cybe}, we consider the triangular case. Triangular \N Lie bialgebras and classical $P$-Nijenhuis Yang-Baxter equations are defined (see Theorem \ref{thm:nqt} and Definition \ref{de:cybe-1}). We introduce the notion of relative Rota-Baxter operators and prove that it can be used to construct \N Lie bialgebras (see Corollary \ref{cor:rbo}), especially on semi-direct product \N Lie algebras (see Theorem \ref{thm:rboliebialg}).

 Finally, in Section \ref{se:nsl2}, we classify all Nijenhuis operators on $\ss$ with respect to the standard Cartan-Weyl basis (Theorem \ref{thm:nsl2}).

 \smallskip
 \noindent{\bf Notations:} Throughout this paper, all vector spaces, tensor products, and linear homomorphisms are over the complex number field $\mathbb{C}$. We denote by $\id_M$ the identity map from $M$ to $M$.

\section{Nijenhuis operators from triangular Lie bialgebras}\label{se:main}
 Subsections \ref{se:qt} and \ref{se:symp} provide the necessary groundwork for the main result established in Subsection \ref{se:n}. We emphasize that Theorem \ref{thm:n} requires a compatibility condition between a solution of the classical Yang-Baxter equation and a solution of the classical co-Yang-Baxter equation.

 \subsection{Triangular Lie bialgebras}\label{se:qt} Let us recall from \cite{Dr} the definition of Lie bialgebra. A {\bf Lie coalgebra} is a pair $(\ll, \d)$, where $\ll$ is a vector space and $\d:\ll\lr \ll\o \ll$ (we use the Sweedler notation \cite{Ra12}: write $\d(x)=x_{(1)}\o x_{(2)}$ ($=\sum x_{(1)}\o x_{(2)}$)) is a linear map such that for all $x\in\ll$,        \vskip-6mm
 \begin{eqnarray}
 &x_{(1)}\o x_{(2)}=-x_{(2)}\o x_{(1)},&\label{eq:coanti}\\
 &x_{(1)}\o x_{(2)(1)}\o x_{(2)(2)}+x_{(2)(1)}\o x_{(2)(2)}\o x_{(1)}+ x_{(2)(2)}\o x_{(1)}\o x_{(2)(1)}=0.&\label{eq:coj}
 \end{eqnarray}

 A {\bf Lie bialgebra} is a triple $(\ll, [,], \d)$, where the pair $(\ll, [,])$ is a Lie algebra, and the pair $(\ll, \d)$ is a Lie coalgebra, such that, for all $x, y\in \ll$,
 \begin{eqnarray}
 &\d([x, y])=[x, y_{(1)}]\o y_{(2)}+ y_{(1)}\o [x, y_{(2)}]-[y, x_{(1)}]\o x_{(2)}-x_{(1)}\o [y, x_{(2)}].& \label{eq:liebialg}
 \end{eqnarray}

 For a given element $r=r^1\o r^2\in \ll\o \ll$, define
 \begin{equation}
 \d(x):=\d_r(x):=r^{1}\o [x, r^{2}]+[x, r^{1}]\o r^{2}, \quad \forall x\in \ll.\label{eq:p}
 \end{equation}

 \begin{rmk}\label{rmk:qt}\cite{Maj} Let $(\ll, [,])$ be a Lie algebra and $r\in \ll\o \ll$. If $\d: \ll\lr  \ll\o \ll$ is given by Eq.(\ref{eq:p}), then $\d$ satisfies Eq.(\ref{eq:liebialg}). $(\ll, \d)$ is a Lie coalgebra if and only if for all $x\in \ll$,
 \begin{eqnarray*}
 &(ad_x\o \id+\id\o ad_x)(r+r^\s)=0,&\\
 &(ad_x\o \id\o \id-\id \o ad_x \o \id-\id\o \id\o ad_x)([r_{12}, r_{13}]+[r_{13}, r_{23}]+[r_{12}, r_{23}])=0,&
 \end{eqnarray*}
 where $[r_{12}, r_{13}]=[r^1, \br^1]\o r^2\o \br^2, [r_{13}, r_{23}]=r^1\o \br^1\o [r^2, \br^2], [r_{12}, r_{23}]=r^1\o [r^2, \br^1]\o \br^2$,
 $\br=r$ and $r^\s=r^2\o r^1$.
 \end{rmk}

 \begin{defi}\cite{Maj} \label{de:rmk:cor:4.12} Let $(\ll, [,])$ be a Lie algebra. If $r\in \ll\o \ll$ is an antisymmetric solution (in the sense of $r=-r^\s$) of the {\bf classical Yang-Baxter equation} (cYBe for short) in $(\ll, [,])$:
 \begin{eqnarray}
 &[r_{12}, r_{13}]+[r_{13}, r_{23}]+[r_{12}, r_{23}]=0,&\label{eq:cybe}
 \end{eqnarray}
 then $(\ll, [,], \d)$ is a Lie bialgebra, where the linear map $\d=\d_r$ is defined by Eq.~\meqref{eq:p}. In this case, we call this Lie bialgebra {\bf triangular} and denoted by $(\ll, [,], r, \d_r)$.
 \end{defi}

 By a direct computation, one can get
 \begin{ex}\label{ex:qt} Let $\l$ be a parameter. All the triangular Lie bialgebraic structures on $(\ss, [,])$ are $(\ss, [,]$, $r, \d_r)$, where
 $$
 r=\l (f\o h-h\o f)
 $$
 and
 \begin{center}$ \left\{
        \begin{array}{l}
        \d_r(e)=2 \l (e\o f-f\o e),\\
        \d_r(f)=0, \\
        \d_r(h)=2 \l (h\o f-f\o h)\\
        \end{array}\right. $.
             \end{center}\smallskip
 \end{ex}

 \subsection{2-cocycle in a Lie algebra}\label{se:symp} In this subsection, we establish that every dual triangular Lie bialgebra naturally induces a 2-cocycle. Furthermore, we classify all such 2-cocycles on the Lie algebra $\ss$.

 \begin{defi}\cite{Maj}\label{de:caybe} Let $(\ll, \d)$ be a Lie coalgebra and $\om\in (\ll\o \ll)^*$. We call the equation below a {\bf classical co-Yang-Baxter equation in $(\ll, \d)$}:
 \begin{equation}\label{eq:caybe}
 \om(x_{(1)}, y)\om(x_{(2)}, z)+\om(x, z_{(1)})\om(y, z_{(2)})+\om(x, y_{(1)})\om(y_{(2)}, z)=0,
 \end{equation}
 where $x, y, z\in \ll$.
 \end{defi}

 \begin{ex}\label{ex:caybe} All the skew-symmetric solutions of the classical co-Yang-Baxter equation in the Lie coalgebra in Example \ref{ex:qt} are given below:
\smallskip
 \begin{center}
  \begin{tabular}{r|rrr}
  $\om$ & $e$  & $f$ & $h$  \\
  \hline
  $e$ & $0$  & $\nu$  & $\kappa$\\
  $f$ & $-\nu$ & $0$ & $\chi$\\
  $h$ &  $-\kappa$ & $-\chi$& $0$ \\
   \end{tabular},
  \end{center}
  where $\nu, \kappa, \chi$ are parameters.
 \end{ex}

 \begin{defi}\cite{Maj} \label{de:cqt} Let $(\ll, \d)$ be a Lie coalgebra and $\om\in (\ll\o \ll)^*$. If $\om$ is a skew-symmetric solution (in the sense of $\om(x, y)=-\om(y, x)$) of the classical co-Yang-Baxter equation, then $(\ll, [,]_\om, \d)$ is a Lie bialgebra, where $[,]_{\om}$ is defined by
 \begin{eqnarray}\label{eq:cp}
 &[x, y]_{\om}=x_{(1)}\om(x_{(2)}, y)-y_{(1)}\om(x, y_{(2)}).&
 \end{eqnarray}
 In this case, we call this Lie bialgebra {\bf dual triangular} and denoted by $(\ll, \d, \om, [,]_{\om})$.
 \end{defi}

 \begin{pro}\label{pro:cqt}\cite{Maj} Let $(\ll, \d)$ be a coalgebra and $\om\in (\ll\o \ll)^*$ skew-symmetric. Then the quadruple $(\ll, \d, \om, [,]_{\om})$, where $[,]_{\om}$ is defined by Eq.(\mref{eq:cp}), is a dual triangular Lie bialgebra if and only if, for all $x, y, z\in \ll$,
 \begin{equation}\label{eq:cqt1}
 \om([x ,y]_{\om}, z)=\om(x, z_{(1)})\om(y, z_{(2)})
 \end{equation}
 or
 \begin{equation}\label{eq:cqt2}
 \om(x, [y, z]_{\om})=-\om(x_{(1)}, y)\om(x_{(2)}, z)
 \end{equation}
 holds.
 \end{pro}

 Next we introduce the notion of a 2-cocycle in a Lie algebra, here we ignore the {\bf nondegeneracy} of symplectic structure given in \cite{BBM}.
 \begin{defi}\label{de:cosymp} Let $(\ll, [,])$ be a Lie algebra and $\omega \in (\ll \o \ll)^*$ be skew-symmetric. If for all $x, y, z \in \ll$, the equation below holds:
 \begin{eqnarray}
 \om([x ,y], z)+\om([y ,z], x)+\om([z ,x], y)=0, \label{eq:cosymp}
 \end{eqnarray}
 then $\om$ is called a {\bf 2-cocycle} in $(\ll, [,])$.
  \end{defi}

 \begin{ex}\label{ex:symp} All the 2-cocycles on $(\ss, [,])$ are given below:
\smallskip
 \begin{center}
  \begin{tabular}{r|rrr}
  $\om$ & $e$  & $f$ & $h$  \\
  \hline
  $e$ & $0$  & $0$  & $\kappa$\\
  $f$ & $0$ & $0$ & $-2\kappa$\\
  $h$ &  $-\kappa$ & $2\kappa$& $0$ \\
   \end{tabular},
  \end{center}
  where $\kappa$ is a parameter.
 \end{ex}

 A 2-cocycle can be obtained by a dual triangular Lie bialgebra.
 \begin{pro}\label{pro:caybetosymp} If $(\ll, \d, \om, [,]_{\om})$, where $[,]_{\om}$ is defined by Eq.(\mref{eq:cp}), is a dual triangular Lie bialgebra, then $\om$ is a 2-cocycle in $(\ll, [,]_{\om})$.
 \end{pro}

 \begin{proof} By Proposition \ref{pro:cqt}, we know that $(\ll, [,]_{\om})$ is a Lie algebra. For all $x, y, z\in \ll$, we have
 \begin{eqnarray*}
 &&\hspace{-13mm}\om([x, y]_{\om}, z)+\om([y, z]_{\om}, x)+\om([z, x]_{\om}, y)\\
 &\stackrel{(\ref{eq:cqt1})}{=}&\om(x, z_{(1)})\om(y, z_{(2)})+\om(y, x_{(1)})\om(z, x_{(2)})+\om(z, y_{(1)})\om(x, y_{(2)})\\
 &\stackrel{(\ref{eq:coanti})}{=}&\om(x_{(1)}, y)\om(x_{(2)}, z)+\om(x, z_{(1)})\om(y, z_{(2)})+\om(x, y_{(1)})\om(y_{(2)}, z)\\
 &&\hspace{10mm} \hbox{~(by the skew-symmetry of } \om)\\
 &\stackrel{(\ref{eq:caybe})}{=}&0,
 \end{eqnarray*}
 as we needed.
 \end{proof}

\subsection{When Lie algebras are Nijenhuis}\label{se:n} A Nijenhuis Lie algebra can be investigated by the following way.

 \begin{thm}\label{thm:n} Let $(\ll, [,])$ be a Lie algebra together with a 2-cocycle $\om$, $r$ be an element in $\ll\o \ll$ such that $(\ll, [,], r, \d_r)$ is a triangular Lie bialgebra. If further, $(\ll, \d_r, \om, [,]_{\om})$ is a dual triangular Lie bialgebra, then $((\ll, [,]), N)$ is a \N Lie algebra, where the \N operator $N: \ll\lr \ll$ is defined by
 \begin{eqnarray}
 &N(x)=\om(x, r^1)r^2,~~\forall~x\in \ll.&\label{eq:ys}
 \end{eqnarray}
 \end{thm}

 \begin{proof} Because $(\ll, \d_r, \om, [,]_{\om})$ is a dual triangular Lie bialgebra, where $\d_r$ is defined by Eq.(\ref{eq:p}), and $\om$ is a 2-cocycle in Lie algebra $(\ll, [,])$, one has
 \begin{eqnarray}\label{eq:73}
 0&=&\om(x_{(1)}, y)\om(x_{(2)}, z)+\om(x, z_{(1)})\om(y, z_{(2)})+\om(x, y_{(1)})\om(y_{(2)}, z)\nonumber\\
 &=&\omega(r^1, y)\omega([x, r^2], z)+\omega([x, r^1], y)\omega(r^2, z)+\omega(x, r^1)\omega(y, [z, r^2])\nonumber\\
 &&+\omega(x, [z, r^1])\omega(y, r^2)+\omega(x, r^1)\omega([y, r^2], z)+\omega(x, [y, r^1])\omega(r^2, z)\nonumber\\
 &=&\omega(r^1, y)\omega([x, r^2], z)+\omega(r^1, y)\omega([r^2, z], x)+\omega(x, r^1)\omega([r^2, z], y)\nonumber\\
 &&+\omega(x, r^1)\omega([y, r^2], z)+\omega([x, r^1], y)\omega(r^2, z)+\omega([r^1, y], x)\omega(r^2, z)\nonumber\\
 &&\hspace{3mm} \hbox{~(by the skew-symmetry of } \om, \hbox{the antisymmetry of } r \hbox{ and } Eq.(\ref{eq:anti}))\nonumber\\
 &=&\omega(y, r^1)\omega([z, x], r^2)+\omega(x, r^1)\omega(r^2, [y, z])+\omega([x, y], r^1)\omega(r^2, z)\\
 &&\hspace{3mm} \hbox{~(by the skew-symmetry of } \om, \hbox{ and } Eq.(\ref{eq:anti})).\nonumber
 \end{eqnarray}

 In the rest, we check that the linear map $N$ defined in Eq.(\mref{eq:ys}) is a Nijenhuis operator on $(\ll, [,])$.
 \begin{eqnarray*}
 &&\hspace{-15mm}[N(x), N(y)]-N([N(x), y]+[x, N(y)]-N([x, y]))\\
 &=&\omega(x,r^1)\omega(y,\bar{r}^1)[r^2,\bar{r}^2]-\omega(x,r^1)\omega([r^2, y],\bar{r}^1)\bar{r}^2-\omega(y,r^1)\omega([x, r^2] ,\bar{r}^1)\bar{r}^2\\
 &&+\omega([x, y],r^1)\omega( r^2 ,\bar{r}^1)\bar{r}^2\\
 &\stackrel{(\ref{eq:73})}{=}&\omega(x,r^1)\omega(y,\bar{r}^1)[r^2,\bar{r}^2]
 -\omega(x,r^1)\omega([r^2, y],\bar{r}^1)\bar{r}^2-\omega(y,r^1)\omega([x, r^2] ,\bar{r}^1)\bar{r}^2\\
 &&+\omega([x, y],r^1)\omega( r^2 ,\bar{r}^1)\bar{r}^2-\omega(x,r^1)\omega([y,\bar{r}^1], r^2)\bar{r}^2-\omega(y,r^1)\omega([\bar{r}^1,x] ,r^2)\bar{r}^2\\
 &&-\omega(r^2,\bar{r}^1)\omega([x,y],r^1)\bar{r}^2\\
 &\stackrel{(\ref{eq:cosymp})}{=}&\omega(x,r^1)\omega(y,\bar{r}^1)[r^2,\bar{r}^2]
 +\omega(x, r^1)\omega([\bar{r}^1,r^2], y)\bar{r}^2+\omega(y, r^1)\omega([r^2,\bar{r}^1], x)\bar{r}^2\\
 &\stackrel{}{=}&\omega(x, r^1)\omega(y, \bar{r}^1)[r^2,\bar{r}^2]+\omega(x,r^1)\omega( y,[r^2,\bar{r}^1])\bar{r}^2+\omega(x, [r^1,\bar{r}^1])\omega(y, r^2)\bar{r}^2 \\
 &&\qquad(\hbox{by the anti-symmetry of $r$})\\
 &\stackrel{(\ref{eq:cybe})}{=}&0.
 \end{eqnarray*}
 Thus, $((\ll, [,]), N)$ is a \N Lie algebra. These finish the proof.
 \end{proof}

 Next we apply Theorem \ref{thm:n} to $\ss$.
 \begin{ex}\label{ex:nsl2} Based on Examples \ref{ex:qt}, \ref{ex:caybe}, \ref{ex:symp}, set $r=\l (f\o h-h\o f)$,\smallskip
 \begin{center}$ \left\{
        \begin{array}{l}
        \d_r(e)=2 \l (e\o f-f\o e),\\
        \d_r(f)=0, \\
        \d_r(h)=2 \l (h\o f-f\o h)\\
        \end{array}\right. $,
             \end{center}\smallskip
 \begin{center}
  \begin{tabular}{r|rrr}
  $\om$ & $e$  & $f$ & $h$  \\
  \hline
  $e$ & $0$  & $0$  & $\kappa$\\
  $f$ & $0$ & $0$ & $-2\kappa$\\
  $h$ &  $-\kappa$ & $2\kappa$& $0$ \\
   \end{tabular},
 \end{center}\smallskip
  then $(\ss, [,], r, \d_r)$ is a triangular Lie bialgebra, $(\ss, \d_r, \om, [,]_{\om})$ is a dual triangular Lie bialgebra and $\om$ is a 2-cocycle in $\ss$. Therefore all the possible Nijenhuis operators $N: \ss\lr \ss$ induced by Theorem \ref{thm:n} on $(\ss, [,])$ are given below:
  \begin{center}
    $\left\{
        \begin{array}{l}
        N(e)=-\l \kappa f,\\
        N(f)=2 \l \kappa f, \\
        N(h)=2 \l \kappa h,\\
        \end{array}\right. $
    \end{center}
   where $\l, \kappa$ are parameters.
 \end{ex}

\section{(Co)Representations of Nijenhuis Lie (co)algebras}\label{se:rep} In this section, we discuss the representation of Nijenhuis Lie algebras and provide a construction of Nijenhuis operators on a Lie coalgebra.

\subsection{Dual representations}\label{se:drep}Let $(\ll, [,])$ be a Lie algebra. A {\bf representation of $(\ll, [,])$} is a pair $(V, \rho)$ where $V$ is a vector space, $\rho: \ll\lr \End(V)$ is a linear map such that, for all $x, y\in \ll$ and $v\in V$,
 \begin{eqnarray}
 &\rho([x, y])(v)=\rho(x)(\rho(y)v)-\rho(y)(\rho(x)v).&\label{eq:replie}
 \end{eqnarray}

 \begin{defi}\label{de:repnlie} If $((\ll, [,]), N)$ is a \N Lie algebra, then a {\bf representation of $((\ll, [,]), N)$} is a triple $((V, \rho), \a)$ consisting of a representation $(V, \rho)$ of $(\ll, [,])$ and a linear map $\a: V\lr V$ such that, for all $x\in \ll$ and $v\in V$,
 \begin{eqnarray}
 &\rho(N(x))\a(v)+\a^2(\rho(x)v)=\a(\rho(N(x))v)+\a(\rho(x)\a(v)).&\label{eq:repnlie}
 \end{eqnarray}
 \end{defi}

 \begin{ex}\label{ex:adrepnlie} Let $((\ll, [,]), N)$ be a \N Lie algebra and $x\in \ll$. Define a linear map $ad: \ll\lr \End(\ll)$ by $ad_x y=[x, y]$, then $((\ll, ad), N)$ is a representation of $((\ll, [,]), N)$, called {\bf adjoint representation of $((\ll, [,]), N)$}.
 \end{ex}

 A representation of a \N Lie algebra can be interpreted by a semi-direct product \N Lie algebra.

 \begin{pro}\label{pro:repnlie} Let $((\ll, [,]), N)$ be a \N Lie algebra, $(V, \rho)$ be a representation of $(\ll, [,])$ and $\a: V\lr V$ a linear map. For all $x, y\in \ll$ and $u, v\in V$, define a bracket on $\ll\oplus V$ by
 \begin{eqnarray}
 &[x+u, y+v]_{\heartsuit}:=[x, y]+\rho(x)v-\rho(y)u&  \label{eq:semip}
 \end{eqnarray}
 and a linear map $N_{\ll\oplus V}: \ll\oplus V\to  \ll\oplus V$ by
 \begin{eqnarray}
 &N_{\ll\oplus V}(x+v):=(N+\a)(x+v)=N(x)+\a(v).&\label{eq:nsemip}
 \end{eqnarray}
 Then $\ll\oplus V$ equipped with the bracket product (\ref{eq:semip}) and the linear map (\ref{eq:nsemip}) turns to be a \N Lie algebra, which is denoted by $(\ll\ltimes^{\rho} V, N+\a)$, if and only if $((V, \rho), \a)$ is a representation of $((\ll, [,]), N)$.
 \end{pro}

 \begin{proof}
 The proof can be seen as a special case of Proposition \ref{pro:mpnlie}, so we skip it.
 \end{proof}

 \begin{lem}\label{lem:dualrep} Let $((\ll, [,]), N)$ be a \N Lie algebra, $(V, \rho)$ be a representation of $(\ll, [,])$ and $\b: V\lr V$ be a linear map. Define a linear map $\rho^*: \ll\lr \End(V^*)$ by
 \begin{eqnarray*}\label{eq:2.8a}
 &\langle \rho^*(x) v^*, v \rangle =-\langle v^*, \rho(x)v \rangle,&
 \end{eqnarray*}
 where $x\in \ll, v^*\in V^*$ and $v\in V$. Then $((V^*, \rho^*), \b^*)$ is a representation of $((\ll, [,]), N)$ if and only if the following condition holds:
 \begin{eqnarray}
 &\b(\rho(N(x))v)+\rho(x)\b^2(v)=\rho(N(x))\b(v)+\b(\rho(x)\b(v)),&\label{eq:dualrep}
 \end{eqnarray}
 where $x\in \ll$ and $v\in V$.

 Especially, for a linear map $P: \ll\lr \ll$, $((\ll^*, ad^*), P^*)$ is a representation of $((\ll, [,]), N)$ if and only if the following condition holds:
 \begin{eqnarray}\label{eq:dualrep-1}
 &P([N(x), y])+[x, P^2(y)]=[N(x), P(y)]+P([x, P(y)]),&
 \end{eqnarray}
 where $x, y\in \ll$.
 \end{lem}

 \begin{proof} First we note that $(V^*, \rho^*)$ is a representation of $(\ll, [,])$ since $(V, \rho)$ is a representation of $(\ll, [,])$. Eq.(\ref{eq:repnlie}) holds for $((V^*, \rho^*), \b^*)$ if and only if, for all $x\in \ll, v\in V, v^*\in V^*$,
 \begin{eqnarray*}
 &&\hspace{-5mm}\langle \rho^*(N(x))\b^*(v^*)+\b^{*2}(\rho^*(x)v^*)-\b^*(\rho^*(N(x))v^*)-\b^*(\rho^*(x)\b^*(v^*)), v \rangle\\
 &&=\langle v^*, -\b(\rho(N(x))v)-\rho(x)\b^2(v)+\rho(N(x))\b(v)+\b(\rho(x)\b(v))\rangle.
 \end{eqnarray*}
 Thus we finish the proof.
 \end{proof}

\begin{defi}\label{de:admop} Under the assumption in Lemma \ref{lem:dualrep}. {\bf $((\ll, [,]), N)$ is $\b$-admissible with respect to $(V, \rho)$} if Eq.(\ref{eq:dualrep}) holds. When Eq.(\ref{eq:dualrep-1}) holds, we say {\bf $((\ll, [,]), N)$ is $P$-admissible}.
 \end{defi}

 \begin{rmk}\label{rmk:dualrep}
  \begin{enumerate}
    \item $((\ll, [,]), N)$ is not $\a$-admissible with respect to $(V, \rho)$ even if $((V, \rho), \a)$ is a representation of $((\ll, [,]), N)$.
    \item $((\ll, [,]), N)$ is $N^*$-admissible.
    \item If $((V, \rho), \a)$ is a representation of $((\ll, [,]), N)$, then when $\b=k \a+\ell \id_V$, $k, \ell\in K$, Eq.(\ref{eq:dualrep}) turns to be
        \begin{eqnarray*}
        &k^2(\rho(x)\a^2(v)-\a(\rho(x)\a(v)))+k(\a^2(\rho(x)v)-\a(\rho(x)\a(v)))+k\ell (\rho(x)\a(v)-\a(\rho(x)v))=0,&
        \end{eqnarray*}
        where $x\in \ll$ and $v\in V$. Therefore $((\ll, [,]), N)$ is $\ell \id_V$-admissible with respect to $(V, \rho)$. If further $\a(\rho(x)v)=\rho(x)\a(v)$, then $((\ll, [,]), N)$ is $\b=k \a+\ell \id_V$-admissible with respect to $(V, \rho)$.
    \item Let $((\ll, [,]), N)$ be a \N Lie algebra, $(V, \rho)$ be a representation of $(\ll, [,])$, $\wn=-\id_\ll-N,~\widetilde{\b}=-\id_V-\b$. Then $((\ll, [,]), \wn)$ is a \N Lie algebra and $((\ll, [,]), N)$ is $\b$-admissible with respect to $(V, \rho)$ if and only if $((\ll, [,]), \wn)$ is $\widetilde{\b}$-admissible with respect to $(V, \rho)$.
 \end{enumerate}
 \end{rmk}

\subsection{Corepresentations of Nijenhuis Lie coalgebras} \label{se:cdrep} Let us consider compatible corepresentations over compatible Lie coalgebras and then introduce the notions of a Nijenhuis Lie coalgebra and its corepresentation.

 \begin{defi}\label{de:corep}\cite{Zh} Let $(\ll, \d)$ be a Lie coalgebra, $V$ be a vector space, $\g: V\lr \ll\o V$ (write $\g(v)=v_{(-1)}\o v_{(0)}$) be a linear map. Then $(V, \g)$ is a corepresentation of $(\ll, \d)$ if
 \begin{eqnarray}
 &v_{(-1)(1)}\o v_{(-1)(2)}\o v_{(0)}=v_{(-1)}\o v_{(0)(-1)}\o v_{(0)(0)}- v_{(0)(-1)}\o v_{(-1)}\o v_{(0)(0)},&\label{eq:corep}
 \end{eqnarray}
 where $v\in V$.

 Let $(V, \g)$ and $(V', \g')$ be corepresentations of $(\ll, \d)$ and $(\ll', \d')$, respectively. A {\bf homomorphism from $(V, \g)$ to $(V', \g')$} is a pair $(f, g)$, where $f: \ll\lr \ll'$ and $g: V\lr V'$ are linear maps such that for all $x\in \ll$ and $v\in V$,
 \begin{eqnarray}
 &\d'(f(x))=(f\o f)(\d(x)),\quad \g'(g(v))=(f\o g)(\g(v)).&\label{eq:homocorep}
 \end{eqnarray}
 \end{defi}

 \begin{lem} \label{lem:comcoalg} Let $\ll$ be a vector spaces, $\d, \D: \ll\lr \ll\o \ll$ (write $\d(x)=x_{(1)}\o x_{(2)}$ and $\D(x)=x_{[1]}\o x_{[2]}$) be two linear maps, $s, t$ be two parameters. Set $\d_{s, t}=s \d+t \D$. Then $(\ll, \d_{s, t})$ is a Lie coalgebra if and only if $(\ll, \d)$, $(\ll, \D)$ are two Lie coalgebras and, for all $x\in \ll$, the equation below holds:
 \begin{eqnarray*}
 &&\hspace{-6mm}x_{(1)}\o x_{(2)[1]}\o x_{(2)[2]}+x_{[1]}\o x_{[2](1)}\o x_{[2](2)}
 +x_{(2)[1]}\o x_{(2)[2]}\o x_{(1)}\nonumber\\
 &&+x_{[2](1)}\o x_{[2](2)}\o x_{[1]}+
 x_{(2)[2]}\o x_{(1)}\o x_{(2)[1]}+x_{[2](2)}\o x_{[1]}\o x_{[2](1)}=0.\label{eq:ca2}
 \end{eqnarray*}
 \end{lem}

 \begin{proof} It is direct by checking Eqs.(\ref{eq:coanti}) and (\ref{eq:coj}) for $\d_{s, t}$.
 \end{proof}

 \begin{lem}\label{lem:comcorep} Let $(\ll, \d_{s, t})$ be a Lie coalgebra, and $\g, \G: V\lr \ll\o V$ (write $\g(v)=v_{(-1)}\o v_{(0)}$ and $\G(v)=v_{[-1]}\o v_{[0]}$) be two linear maps. Set $\g_{s, t}=s \g+t \G$. Then $(V, \g_{s, t})$ is a corepresentation of $(\ll, \d_{s, t})$ if and only if $(V, \g)$, $(V, \G)$ are corepresentations of $(\ll, \d)$ and $(\ll, \D)$, respectively, and further, for all $v\in V$, the equation below holds:
 \begin{eqnarray*}
 &&\hspace{-6mm}v_{(-1)[1]}\o v_{(-1)[2]}\o v_{(0)}+v_{[-1](1)}\o v_{[-1](2)}\o v_{[0]}\\
 &&=v_{(-1)}\o v_{(0)[-1]}\o v_{(0)[0]}+v_{[-1]}\o v_{[0](-1)}\o v_{[0](0)}\\
 &&\quad-v_{(0)[-1]}\o v_{(-1)}\o v_{(0)[0]}-v_{[0](-1)}\o v_{[-1]}\o v_{[0](0)}.\label{eq:ca2}
 \end{eqnarray*}
 \end{lem}

 \begin{proof} It is direct by checking Eq.(\ref{eq:corep}) for $\d_{s, t}$ and $\g_{s, t}$.
 \end{proof}

\begin{thm}\label{thm:comcn} Let $(V, \g)$ be a corepresentation of $(\ll, \d)$, $(V, \g_{s, t})$ be a corepresentation of $(\ll, \d_{s, t})$, $P: \ll\lr \ll$ and $\b: V\lr V$ be two linear maps. Then $(s\id_\ll+t P, s\id_V+t \b)$ is a homomorphism from $(V, \g)$ of $(\ll, \d)$ to $(V, \g_{s, t})$ of $(\ll, \d_{s, t})$ if and only if for all $x\in \ll$ and $v\in V$, the equations below hold:
 \begin{eqnarray}
 &P(x)_{[1]}\o P(x)_{[2]}=P(x_{(1)})\o P(x_{(2)}),&\label{eq:cmh1}\\
 &P(x)_{(1)}\o P(x)_{(2)}+x_{[1]}\o x_{[2]}=x_{(1)}\o P(x_{(2)})+P(x_{(1)})\o x_{(2)},&\label{eq:cmh2}\\
 &\b(v)_{[-1]}\o \b(v)_{[0]}=P(v_{(-1)})\o \b(v_{(0)}),&\label{eq:cmh3}\\
 &\b(v)_{(-1)}\o \b(v)_{(0)}+v_{[-1]}\o v_{[0]}=v_{(-1)}\o \b(v_{(0)})+P(v_{(-1)})\o v_{(0)}.&\label{eq:cmh4}
 \end{eqnarray}
 \end{thm}

 \begin{proof} It is straightforward by checking Eq.(\ref{eq:homocorep}) for $(s\id_\ll+t P, s\id_V+t \b)$.
 \end{proof}

 Based on Eqs.(\ref{eq:cmh1}) and (\ref{eq:cmh2}), we have

\begin{defi}\label{de:nc} Let $(\ll, \d)$ be a Lie coalgebra and $P:\ll\lr \ll$ a linear map. A {\bf \N Lie coalgebra} is a triple $((\ll, \d), P)$ such that for all $x\in \ll$,
 \begin{eqnarray}
 &P(x_{(1)})\o P(x_{(2)})+P^2(x)_{(1)}\o P^2(x)_{(2)}=P(P(x)_{(1)})\o P(x)_{(2)}+P(x)_{(1)}\o P(P(x)_{(2)}).&\label{eq:nc}
 \end{eqnarray}
 \end{defi}

 For a linear map $\psi: M\to  V$, we denote the transpose map by $\psi^*: V^*\lr  M^*$ given by
 $$
 \langle \psi^*(v^*), m \rangle=\langle v^*, \psi(m) \rangle , \quad \forall~ m\in M, v^*\in V^*.
 $$

 \begin{pro}\label{pro:dualacoa}
  \begin{enumerate}
    \item \label{it:dualaco1} If $((\ll, \d), P)$ is a \N Lie coalgebra, then $((\ll^*, \d^*), P^*)$ is a \N Lie algebra, where $\d^*:\ll^*\o \ll^*\lr \ll^*$ is defined by $\langle \d^*(x^*\o y^*),  z\rangle=\langle x^*,  z_{(1)}\rangle \langle y^*,  z_{(2)}\rangle$.
    \item \label{it:dualaco2} If $((\ll, [,]), N)$ is a finite dimensional \N Lie algebra, then $((\ll^*, [,]^*), N^*)$ is a \N Lie coalgebra, where $[,]^*:\ll^*\lr \ll^*\o \ll^*$ is defined by $\langle [,]^*(x^*),  y\o z\rangle=\langle x^*,  [y, z]\rangle$.
 \end{enumerate}
 \end{pro}

 \begin{proof} We only verify Item (\ref{it:dualaco1}). Item (\ref{it:dualaco2}) is left to the readers. Since $(\ll, \d)$ is a Lie coalgebra, then $(\ll^*, \d^*)$ is a Lie algebra (\cite{Maj}). Next we only need to check that Eq.(\ref{eq:n}) holds for $\d^*, P^*$. In fact, for all $x^*, y^*\in \ll^*$ and $z\in \ll$, one gets
 \begin{eqnarray*}
 &&\hspace{-6mm}\langle \d^*(P^*(x^*)\o P^*(y^*))+P^{*2}(\d^*(x^*\o y^*))-P^*(\d^*(P^*(x^*)\o y^*))-P^*(\d^*(x^*\o P^*(y^*))), z\rangle\\
 &&\hspace{-3mm}=\langle x^*\o y^*,  P(z_{(1)})\o P(z_{(2)})+P^2(z)_{(1)}\o P^2(z)_{(2)}-P(P(z)_{(1)})\o P(z)_{(2)}-P(z)_{(1)}\o P(P(z)_{(2)})\rangle.
 \end{eqnarray*}
 We finish the proof.
 \end{proof}

  According to Eqs.(\ref{eq:cmh3}) and (\ref{eq:cmh4}), one can obtain

 \begin{defi}\label{de:ncm} Let $((\ll, \d), P)$ be a \N Lie coalgebra, $(V, \g)$ be a corepresentation of $(\ll, \d)$ and $\b: V\lr V$ a linear map. A {\bf corepresentation of $((\ll, \d), P)$} is a triple $((V, \g), \b)$ such that for all $v\in V$, the equation below holds:
 \begin{eqnarray*}
 &P(v_{(-1)})\o \b(v_{(0)})+\b^2(v)_{(-1)}\o \b^2(v)_{(0)}=P(\b(v)_{(-1)})\o \b(v)_{(0)}+\b(v)_{(-1)}\o \b(\b(v)_{(0)}).&\label{eq:ncom}
 \end{eqnarray*}
 \end{defi}

 \begin{pro}\label{pro:semicop} Let $((\ll, \d), P)$ be a Nijenhuis Lie coalgebra, $V$ a vector space, $\g: V\lr \ll\o V$, $\b: V\lr V$ two linear maps. Define a coproduct
 \begin{eqnarray*}
 &\d_{\ll\oplus V}(x+v):=x_{(1)}\o x_{(2)}+v_{(-1)}\o v_{(0)}-v_{(0)}\o v_{(-1)},~\forall~x\in \ll, ~v\in V&\label{eq:semicop}
 \end{eqnarray*}
 on $\ll\oplus V$ and a linear map $P_{\ll\oplus V}$ on $\ll\oplus V$ by $P_{\ll\oplus V}(x+v):=P(x)+\b(v)$. Then $((\ll\oplus V, \d_{\ll\oplus V}), P_{\ll\oplus V})$ is a \N Lie coalgebra if and only if $((V, \g), \b)$ is a corepresentation of $((\ll, \d), P)$. In this case, we call this Nijenhuis Lie coalgebra {\bf semi-direct coproduct} of $((\ll, \d), P)$ by its corepresentation $((V, \g), \b)$ and denote it by $(\ll\times_{\g} V, P+\b)$.
 \end{pro}

 \begin{proof} First we note that $(\ll\oplus V, \d_{\ll\oplus V})$ is a Lie coalgebra if and only if $(V, \g)$ is a corepresentation of $(\ll, \d)$ (\cite{Maj}). For all $x\in \ll$ and $v\in V$, Eq.(\ref{eq:nc}) holds for $\d_{\ll\oplus V}$ and $P_{\ll\oplus V}$ if and only if
 \begin{eqnarray*}
 &&P(x_{(1)})\o P(x_{(2)})+P(v_{(-1)})\o \b(v_{(0)})-\b(v_{(0)})\o P(v_{(-1)})\\
 &&+P^2(x)_{(1)}\o P^2(x)_{(2)}+\b^2(v)_{(-1)}\o \b^2(v)_{(0)}-\b^2(v)_{(0)}\o \b^2(v)_{(-1)}\\
 &&\hspace{-3mm}=P(P(x)_{(1)})\o P(x)_{(2)}+P(\b(v)_{(-1)})\o \b(v)_{(0)}-\b(\b(v)_{(0)})\o \b(v)_{(-1)}\\
 &&+P(x)_{(1)}\o P(P(x)_{(2)})+\b(v)_{(-1)}\o \b(\b(v)_{(0)})-\b(v)_{(0)}\o P(\b(v)_{(-1)}).
 \end{eqnarray*}
 Thus we can finish the proof.\end{proof}

 Proposition \ref{pro:dualacoa} can be extended to

\begin{pro}\label{pro:dualncrep}
  \begin{enumerate}
    \item \label{it:dualncrep1} If $((V, \g), \b)$ is a corepresentation of $((\ll, \d), P)$, then $((V^*, \g^*), \b^*)$ is a representation of $((\ll^*, \d^*), P^*)$, where $\g^*: \ll^*\o V^*\lr V^*$ is defined by $\langle \g^*(x^*\o v^*),  u\rangle=\langle x^*,  u_{(-1)}\rangle \langle v^*,  u_{(0)}\rangle$.
    \item \label{it:dualncrep2} If $((V, \rho), \a)$ is a finite dimensional representation of a finite dimensional \N Lie algebra $((\ll, [,]), N)$, then $((V^*, \rho^*), \a^*)$ is a corepresentation of $((\ll^*, [,]^*), N^*)$, where $\rho^*: V^*\lr \ll^*\o V^*$ is defined by $\langle \rho^*(v^*),  x\o u\rangle=\langle v^*,  \rho(x)u \rangle$.
 \end{enumerate}
 \end{pro}

 \subsection{A sufficient condition for Nijenhuis Lie coalgebra}\label{se:con} In this subsection, we consider the construction of \N operators on a Lie coalgebra. First we give the dual concept of a 2-cocycle in a Lie algebra.

 \begin{defi} Let $(\ll,\d)$ be a Lie coalgebra and $r\in \ll\o \ll$ be an antisymmetric element. Assume that
 \begin{equation*}
 r^1{_{1}}\o r^1{_{2}}\o r^2+r^2\o r^1{_{1}}\o r^1{_{2}}+r^1{_{2}}\o r^2\o r^1{_{1}}=0.
 \end{equation*}
 Then $r$ is called a {\bf 2-cycle} in a Lie coalgebra $(\ll, \d)$.
 \end{defi}

 \begin{pro} Let $(\ll,\d)$ be a Lie coalgebra and $r\in \ll\o \ll$ be an antisymmetric element. If $(\ll, [,], r, \d_r)$ is a triangular Lie bialgebra, where $\d_r$ is defined by Eq.$(\ref{eq:p})$, then $r$ is a 2-cycle in the Lie coalgebra $(\ll, \d_r)$.
 \end{pro}
 \begin{proof}
 It is direct by Eqs.(\ref{eq:p}) and (\ref{eq:cybe}).
 \end{proof}

 \begin{thm}\label{thm:con} Let $(\ll, \d)$ be a Lie coalgebra together with a 2-cycle $r$ and $\om\in(\ll\o \ll)^*$. If $(\ll, \d, \om, [,]_{\om})$ is a dual triangular Lie bialgebra with the multiplication $[,]_{\om}$ defined in Eq.(\ref{eq:cp}) and further, $(\ll, [,]_{\om}, r, \d_r)$ is a triangular Lie bialgebra together with the comultiplication $\d_r$ defined in Eq.(\ref{eq:p}), then $((\ll, \d), P)$ is a Nijenhuis Lie coalgebra, where $P$ is defined by
 \begin{equation*}
 P(x)=r^1\om(r^2, x),~~\forall~x\in \ll.
 \end{equation*}
 \end{thm}
 \begin{proof}
 Similar to Theorem \ref{thm:n}.
 \end{proof}

\section{Nijenhuis Lie bialgebras}\label{se:nliebialg} In this section, we establish the bialgebraic theory of \N Lie algebras and introduce the notion of classical $P$-Nijenhuis Yang-Baxter equations. That is to say, we consider the following commutative diagram. \bigskip
{\small
\begin{center}
\unitlength 1mm 
\linethickness{0.4pt}
\ifx\plotpoint\undefined\newsavebox{\plotpoint}\fi 
\hspace{-3.5cm}\begin{picture}(83.5,34)(0,0)
\put(7,6){Lie algebras}
\put(69,6.5){\N Lie algebras.}
\put(3,31.75){Lie bialgebras}
\put(69,30.75){\N Lie bialgebras}
\put(83,12.5){\vector(0,1){13.5}}
\put(17,12.5){\vector(0,1){13.5}}
\put(31.75,34){+Nijenhuis operator}
\put(31.75,8.75){+Nijenhuis operator}
\put(83.5,18.5){Bialgebraization}
\put(18,18.75){Bialgebraization}
\put(29.5,32){\vector(1,0){36.25}}
\put(29.5,6.75){\vector(1,0){36.25}}
\end{picture}
\end{center}
} 

 \subsection{Nijenhuis Lie bialgebras and equivalent characterizations}\label{se:equi}
  
\begin{defi}\label{de:nliebialg} A {\bf \N Lie bialgebra} is a quintuple $((\ll, [,]), \d, N, P)$, where
  \begin{enumerate}
    \item \label{it:nliebialg1} $((\ll, [,]), N)$ is \N Lie algebra;
    \item \label{it:nliebialg2} $((\ll, \d), P)$ is \N Lie coalgebra;
    \item \label{it:nliebialg3} $(\ll, [,], \d)$ is Lie bialgebra;
    \item \label{it:nliebialg4} For all $x, y\in \ll$, the equations below hold:
 \begin{eqnarray}
 &P([N(x), y])+[x, P^2(y)]=[N(x), P(y)]+P([x, P(y)]),&\label{eq:nliebialg1}\\
 &P(N(x)_{(1)})\o N(x)_{(2)}+x_{(1)}\o N^2(x_{(2)})=P(x_{(1)})\o N(x_{(2)})+N(x)_{(1)}\o N(N(x)_{(2)}).&\label{eq:nliebialg2}
 \end{eqnarray}
 \end{enumerate}
 \end{defi}

 \begin{rmk}\label{rmk:de:nliebialg}
  \begin{enumerate}
    \item Eq.(\ref{eq:nliebialg1}) means that $((\ll, [,]), N)$ is $P$-admissible;
    \item Eq.(\ref{eq:nliebialg2}) means that $((\ll^*, \d^*), P^*)$ is $N^*$-admissible.
 \end{enumerate}
 \end{rmk}

 \begin{ex}\label{ex:nliebialg} Let $\ll=\ss$ be a Lie algebra given above. Set\smallskip
 \begin{center}$ \left\{
        \begin{array}{l}
        \d(e)=2 \l (e\o f-f\o e),\\
        \d(f)=0, \\
        \d(h)=2 \l (h\o f-f\o h)\\
        \end{array}\right. $.
             \end{center}\smallskip
        Then by Example \ref{ex:qt}, $(\ll, \d)$ is a Lie coalgebra and further $(\ll, [,], \d)$ is a Lie bialgebra. Define the linear maps $N, P:\ll\lr \ll$ as follows ($k_1, k_2, k_3$ and $\ell$ are parameters).
   \begin{enumerate}
    \item \label{it:ex:nliebialg-1} $\left\{
        \begin{array}{l}
         N(e)=k_1 e+k_2 f\\
         N(f)=k_1 f\\
         N(h)=k_1 h\\
         \end{array}
           \right.$,\quad
         $P=k_1 \id $;
    \item \label{it:ex:nliebialg-2}  $\left\{
        \begin{array}{l}
         N(e)=k_1 e+k_2 f\\
         N(f)=k_1 f\\
         N(h)=k_1 h\\
         \end{array}
           \right.$,\quad $ \left\{
        \begin{array}{l}
         P(e)=k_1 e+\ell f\\
         P(f)=k_1 f\\
         P(h)=k_1 h\\
         \end{array}
           \right.$;
    \item \label{it:ex:nliebialg-3}   $ \left\{
        \begin{array}{l}
         N(e)=k_1 e+k_2 f+k_3 h\\
         N(f)=k_1 f\\
         N(h)=k_1 h\\
         \end{array}
           \right.$,\quad $ \left\{
        \begin{array}{l}
         P(e)=k_1 e-k_2 f\\
         P(f)=k_1 f\\
         P(h)=k_1 h\\
         \end{array}
           \right.$;
    \item \label{it:ex:nliebialg-4}   $ \left\{
        \begin{array}{l}
         N(e)=k_1 e+k_2 f+k_3 h\\
         N(f)=k_1 f\\
         N(h)=k_1 h\\
         \end{array}
           \right.$,\quad $ \left\{
        \begin{array}{l}
         P(e)=k_1 e-k_2 f-k_3 h\\
         P(f)=k_1 f\\
         P(h)=k_1 h\\
         \end{array}
           \right.(k_2\neq 0, k_3\neq 0)$.
 \end{enumerate}
 Then by a direct computation we can find that $((\ll, [,]), \d, N, P)$ is a \N Lie bialgebra, where the $(N, P)$ is given in Item (\ref{it:ex:nliebialg-1})-Item (\ref{it:ex:nliebialg-4}), respectively.
 \end{ex}

 \N Lie bialgebra can be characterized by Manin triple of \N Lie algebras and matched pair of Lie algebras.

\begin{defi}\label{de:mpnlie} A {\bf matched pair of Nijenhuis Lie algebras $((\ll, [,]_\ll), N_\ll)$ and $((\hh, [,]_\hh), N_\hh)$} is an octuple $(((\ll, [,]_\ll), N_\ll), ((\hh, [,]_\hh), N_\hh), \rho_\ll, \rho_\hh)$, where $((\hh, \rho_\ll), N_\hh)$ is a representation of $((\ll, [,]_\ll), N_\ll)$, $((\ll, \rho_\hh), N_\ll)$ is a representation of $((\hh, [,]_\hh), N_\hh)$, and $((\ll, [,]_\ll), (\hh, [,]_\hh), \rho_\ll, \rho_\hh)$ is a matched pair of Lie algebras $(\ll, [,]_\ll)$ and $(\hh, [,]_\hh)$, that is to say, the equations below hold:
 \begin{eqnarray}
 &\rho_\hh(a)[x, y]_\ll-[\rho_\hh(a)x, y]_\ll-[x, \rho_\hh(a)y]_\ll=\rho_\hh(\rho_\ll(y)a)x-\rho_\hh(\rho_\ll(x)a)y,& \label{eq:mplie1}\\
 &\rho_\ll(x)[a, b]_\hh-[\rho_\ll(x)a, b]_\hh-[a, \rho_\ll(x)b]_\hh=\rho_\ll(\rho_\hh(b)x)a-\rho_\ll(\rho_\hh(a)x)b,&\label{eq:mplie2}
 \end{eqnarray}
 where $x, y\in \ll$ and $a, b\in \hh$.
 \end{defi}

 \begin{pro}\label{pro:mpnlie} Let $((\ll, [,]_\ll), N_\ll)$ and $((\hh, [,]_\hh), N_\hh)$ be two Nijenhuis Lie algebras. Assume that $((\hh, \rho_\ll), N_\hh)$ is a representation of $((\ll, [,]_\ll), N_\ll)$ and $((\ll, \rho_\hh), N_\ll)$ is a representation of $((\hh, [,]_\hh), N_\hh)$. Define a bracket $[,]_{\star}$ on $\ll\oplus \hh$ by
 \begin{eqnarray}
 &[x+a, y+b]_{\star}=[x, y]_\ll+\rho_\hh(a)y-\rho_\hh(b)x+[a, b]_\hh+\rho_\ll(x)b-\rho_\ll(y)a&\label{eq:mpnliep}
 \end{eqnarray}
 and a linear map $N_\star$ on $\ll\oplus \hh$ by
 \begin{eqnarray}
 &N_\star(x+a)=N_\ll(x)+N_\hh(a),&\label{eq:mpnlien}
 \end{eqnarray}
 where $x, y\in \ll$ and $a, b\in \hh$. Then $((\ll\oplus \hh, [,]_{\star}), N_\star)$ is a \N Lie algebra if and only if Eqs.(\ref{eq:mplie1}) and (\ref{eq:mplie2}) hold. In this case, we denote $((\ll\oplus \hh, [,]_{\star}), N_\star)$ by $(\ll\bowtie \hh, N_\ll+N_\hh)$.
 \end{pro}

 \begin{proof} We note that $(\ll\oplus \hh, [,]_{\star})$ is a Lie algebra if and only if Eqs.(\ref{eq:mplie1}) and (\ref{eq:mplie2}) hold (\cite{Maj}). So we only need to check that $N_\star$ satisfies Eq.(\ref{eq:n}). In fact, For all $x, y\in \ll$ and $a, b\in \hh$, we have
 \begin{eqnarray*}
 &&\hspace{-6mm}[N_\ll(x), N_\ll(y)]_\ll+\rho_\hh(N_\hh(a))N_\ll(y)-\rho_\hh(N_\hh(b))N_\ll(x)+[N_\hh(a), N_\hh(b)]_\hh+\rho_\ll(N_\ll(x))N_\hh(b)\\
 &&-\rho_\ll(N_\ll(y))N_\hh(a)+N_\ll^2([x, y]_\ll)+N_\ll^2(\rho_\hh(a)y)-N_\ll^2(\rho_\hh(b)x)+N_\hh^2([a, b]_\hh)+N_\hh^2(\rho_\ll(x)b)\\
 &&-N_\hh^2(\rho_\ll(y)a)\\
 &&\hspace{-6mm}\stackrel{(\ref{eq:n})(\ref{eq:repnlie})}{=}N_\ll([N_\ll(x), y]_\ll)+N_\ll(\rho_\hh(N_\hh(a))y)-N_\ll(\rho_\hh(b)N_\ll(x))+N_\hh([N_\hh(a), b]_\hh)+N_\hh(\rho_\ll(N_\ll(x))b)\\
 &&-N_\hh(\rho_\ll(y)N_\hh(a))+N_\ll([x, N_\ll(y)]_\ll)+N_\ll(\rho_\hh(a)N_\ll(y))-N_\ll(\rho_\hh(N_\hh(b))x)+N_\hh([a, N_\hh(b)]_\hh)\\
 &&+N_\hh(\rho_\ll(x)N_\hh(b))-N_\hh(\rho_\ll(N_\ll(y))a).
 \end{eqnarray*}
 Then  we finish the proof.
 \end{proof}

 \begin{rmk} \label{rmk:pro:mpnlie} If $(((\ll, [,]_\ll), N_\ll), ((\hh, [,]_\hh), N_\hh), \rho_\ll, \rho_\hh)$ is a matched pair of Nijenhuis Lie algebras $((\ll, [,]_\ll)$, $N_\ll)$ and $((\hh, [,]_\hh), N_\hh)$, then one can obtain a Nijenhuis Lie algebra $((\ll\oplus \hh, [,]_{\star}), N_\star)$.
 \end{rmk}

 \begin{thm} \label{thm:matchandnlieb} Let $((\ll, [,]), N)$ and $((\ll^*, [,]_\circ:=\d^*), P^*)$ be two \N Lie algebras. Then the quintuple $((\ll, [,]), \d, N, P)$ is a \N Lie bialgebra if and only if $(((\ll, [,]), N), ((\ll^*,$ $[,]_\circ), P^*), {ad_\ll}^*$, ${ad_{\ll^*}}^*)$ is a matched pair of Nijenhuis Lie algebras $((\ll, [,]), N)$ and $((\ll^*, [,]_\circ), P^*)$.
 \end{thm}

 \begin{proof} For Lie algebras, we have (\cite{Maj}): $((\ll, [,]), \d)$ is a Lie bialgebra if and only if $((\ll, [,]),$ $(\ll^*, [,]_\circ)$, ${ad_\ll}^*, {ad_{\ll^*}}^*)$ is a matched pair of Lie algebras $(\ll, [,])$ and $(\ll^*, [,]_\circ)$, where Eq.(\ref{eq:mplie1})=Eq.(\ref{eq:mplie2}) $\Leftrightarrow$ Eq.(\ref{eq:liebialg}). The rest can be obtained by Remark \ref{rmk:de:nliebialg}.
 \end{proof}

 Let $(\ll, [, ])$ be a Lie algebra, $\frakB: \ll\o \ll\lr K$ be a nondegenerate bilinear form. If further $\frakB$ is invariant, i.e., for all $x, y, z\in \ll$, $\frakB([x, y], z)=\frakB(x, [y, z])$. In this case, the triple $((\ll, [,]), \frakB)$ is called a {\bf Frobenius Lie algebra}. A Frobenius algebra $((\ll, [,]), \frakB)$ is {\bf symmetric} if $\frakB$ is symmetric.
 \begin{defi}\label{de:nflie} A {\bf Nijenhuis  Frobenius Lie algebra} is a quadruple $(((\ll, [,]), N), \frakB)$ consisting of a Nijenhuis Lie algebra $((\ll, [,]), N)$ and a Frobenius Lie algebra $((\ll, [,]), \frakB)$.

 Let $\hat{N}: \ll\lr \ll$ denote the adjoint operator of $N$ by the following way:
 \begin{eqnarray*}
 &\mathfrak{B}(N(x), y)=\mathfrak{B}(x, \hat N(y)),\;\;\forall x, y\in \ll.&\label{eq:adjoint}
 \end{eqnarray*}
 \end{defi}

 \begin{pro}\label{pro:frobadm} Let $(((\ll, [,]), N), \frakB)$ be a symmetric Nijenhuis Frobenius Lie algebra. Then $((\ll, [,])$, $N)$ is ${\hat N}$-admissible.
 \end{pro}

 \begin{proof} It is natural that $(\ll^*, ad^*)$ is a representation of $(\ll, [,])$. In what follows, we only need to prove that Eq.(\ref{eq:dualrep-1}) holds for $ad$ and ${\hat N}$. In fact, for all $x, y, z\in \ll$,
 \begin{eqnarray*}
 0&\stackrel{(\ref{eq:n})}{=}&\frakB([N(x), N(y)], z)+\frakB(N^2([x, y]), z)-\frakB(N([x, N(y)]), z)-\frakB(N([N(x), y]), z)\\
 &=&\frakB(N(x), [N(y), z])+\frakB([x, y], {\hat N}^2(z))-\frakB([x, N(y)], {\hat N}(z))-\frakB([N(x), y], {\hat N}(z))\\
 &=&\frakB(x, {\hat N}([N(y), z]))+\frakB(x, [y, {\hat N}^2(z)])-\frakB(x, [N(y), {\hat N}(z)])-\frakB(x, {\hat N}([y, {\hat N}(z)])),
 \end{eqnarray*}
 so the proof is finished.
 \end{proof}

\begin{defi} \label{de:manintriple} Let $((\ll, [,]), N)$ and $((\ll^*, [,]_\circ), P^*)$ be two \N Lie algebras. Assume that there is a Lie algebra structure $[,]_\star$ on $\ll\oplus \ll^*$ which contains both $(\ll, [,])$ and $(\ll^*, [,]_{\circ})$ as Lie subalgebras. Define a bilinear form on $\ll\oplus \ll^*$ by
 \begin{eqnarray}
 &\frakB_d(x+a^*, y+b^*)=\langle a^*, y\rangle+\langle b^*, x\rangle,~\forall~x, y\in \ll, a^*, b^*\in \ll^*.&\label{eq:biform}
 \end{eqnarray}
 If $\frakB_d$ is invariant, then $((\ll\oplus \ll^*, [,]_\star), \frakB_d)$ is a symmetric Frobenius Lie algebra. We call this Frobenius Lie algebra a {\bf Manin triple of a Frobenius Lie algebra associated to $(\ll, [,])$ and $(\ll^*, [,]_{\circ})$} and denote it by $(\ll\bowtie \ll^*, \frakB_d)$. If further $((\ll\bowtie \ll^*, N+P^*), \frakB_d)$ is a Nijenhuis Frobenius Lie algebra, then we call $((\ll\bowtie \ll^*, N+P^*), \frakB_d)$ a {\bf Manin triple of a Nijenhuis Frobenius Lie algebra associated to $((\ll, [,]), N)$ and $((\ll^*, [,]_{\circ}), P^*)$}.
 \end{defi}

 \begin{lem}\label{lem:abas} Let $((\ll\bowtie \ll^*, N+P^*), \frakB_d)$ be a Manin triple of a Nijenhuis Frobenius Lie algebra associated to $((\ll, [,]), N)$ and $((\ll^*, [,]_{\circ}), P^*)$. Then
 \begin{enumerate}
   \item The adjoint $\widehat{N+P^*}$ of $N+P^*$ with respect to $\frakB_d$ is $P+ N^*$. Further $(\ll\bowtie \ll^*, N+P^*)$ is $(P+N^*)$-admissible.\label{it:abas1}
   \item $((\ll, [,]), N)$ is $P$-admissible.\label{it:abas2}
   \item $((\ll^*, [,]_\circ), P^*)$ is $N^*$-admissible.\label{it:abas3}
 \end{enumerate}
 \end{lem}

 \begin{proof} (\ref{it:abas1}) For all $x, y\in \ll$ and $a^*, b^*\in \ll^*$, we have
 \begin{eqnarray*}
 \frakB((N+P^*)(x+a^*), y+b^*)
 &=&\langle b^*, N(x)\rangle+\langle P^*(a^*), y\rangle\\
 &=&\langle N^*(b^*), x\rangle+\langle a^*, P(y)\rangle=\frakB(x+a^*, (P+N^*)(y+b^*)).
 \end{eqnarray*}
 Then by Proposition \ref{pro:frobadm}, the proof of Item (\ref{it:abas1}) is finished.

 (\ref{it:abas2}) Since $(\ll\bowtie \ll^*, N+P^*)$ is $(P+N^*)$-admissible, by Eq.(\ref{eq:dualrep-1}), for $x, y\in \ll, a^*, b^*\in \ll^*$, one has
 \begin{eqnarray}
 &&(P+N^*)[(N(x)+P^*(a^*)), (y+b^*)]+[(x+a^*), (P^2(y)+N^{*2}(b^*))]\nonumber\\
 &&=[(N(x)+P^*(a^*)), (P(y)+N^*(b^*))]+(P+N^*)[(x+a^*),(P(y)+N^*(b^*))]. \label{eq:dualrep-2}
 \end{eqnarray}
 Then Item~(\ref{it:abas2}) can be obtained by lettimg $a^*=b^*=0$ in Eq. (\ref{eq:dualrep-2}).

 (\ref{it:abas3}) Similarly, Item~(\ref{it:abas3}) can be proved by setting $x=y=0$ in Eq. (\ref{eq:dualrep-2}).
 \end{proof}

 \begin{thm}\label{thm:matchandmanin} Let $((\ll, [,]), N)$ and $((\ll^*, [,]_\circ), P^*)$ be two \N Lie algebras. Then there is a Manin triple $((\ll\bowtie \ll^*, N+P^*), \frakB_d)$ of a Nijenhuis Frobenius Lie algebra associated to $((\ll, [,]), N)$ and $((\ll^*, [,]_{\circ})$, $P^*)$ if and only if $(((\ll, [,]), N), ((\ll^*, [,]_\circ), P^*), {ad_\ll}^*, {ad_{\ll^*}}^*)$ is a matched pair of Nijenhuis Lie algebras $((\ll, [,]), N)$ and $((\ll^*, [,]_\circ), P^*)$.
 \end{thm}

 \begin{proof} For Lie algebras, we have (\cite{Maj}): there is a Manin triple $(\ll\bowtie \ll^*, \frakB_d)$ of a Frobenius Lie algebra associated to $(\ll, [,])$ and $(\ll^*, [,]_{\circ})$ if and only if $((\ll, [,]), (\ll^*, [,]_\circ), {ad_\ll}^*, {ad_{\ll^*}}^*)$ is a matched pair of Lie algebras $(\ll, [,])$ and $(\ll^*, [,]_\circ)$. The rest can be obtained by Proposition \ref{pro:mpnlie} and Lemma \ref{lem:abas}.
 \end{proof}

 \subsection{Enrichment of classical Yang-Baxter equations}\label{se:cybe}
 Let $((\ll, [,]), N)$ be a $P$-admissible Nijenhuis Lie algebra, $r$ be an antisymmetric solution of the cYBe in $(\ll, [,])$. Then by Remark \ref{rmk:qt}, we can get a triangular Lie bialgebra $(\ll, [,], r, \d_r)$. Thus in this case, by Remark \ref{rmk:de:nliebialg}, $((\ll, [,]), \d_r, N, P)$ is a \N Lie bialgebra if and only if Eqs.(\ref{eq:nc}) and (\ref{eq:nliebialg2}) hold. In what follows, we continue to study these two conditions by using the element $r$. Eqs.(\ref{eq:nc-1}) and (\ref{eq:nliebialg2-1})
 can be seen as the compatible conditions between the cYBe and Nijenhuis operators $N$ on
 Lie algebra and $P$ on Lie coalgebra.
 \begin{thm} \label{thm:pq} Let $((\ll, [,]), N)$ be a $P$-admissible Nijenhuis Lie algebra and $\d:=\d_r$ given by Eq.(\ref{eq:p}). Then
 \begin{enumerate}
   \item \label{it:pq1} Eq.(\ref{eq:nc}) holds  if and only if, for all $x\in \ll$,
    \begin{eqnarray}
     &&(\id\o P\circ ad_x-\id\o ad_{P(x)})((P\o \id-\id\o N)(r)) \label{eq:nc-1}\\
     &&\qquad+(ad_{P(x)}\o \id-P\circ ad_x\o \id)((N\o \id-\id\o P)(r))=0.\nonumber
    \end{eqnarray}
   \item \label{it:pq2} Eq.(\ref{eq:nliebialg2}) holds if and only if, for all $x\in \ll$,
    \begin{eqnarray}
     &&(\id\o ad_{N(x)}-\id\o N\circ ad_x+ad_{N(x)}\o \id+P\circ ad_x\o \id)((P\o \id-\id\o N)(r))\label{eq:nliebialg2-1}\\
     &&\hspace{75mm}=(ad_x\circ P^2\o \id-ad_x\o N^2)(r).\nonumber
    \end{eqnarray}
 \end{enumerate}
 \end{thm}

 \begin{proof} (\ref{it:pq1}) For all $x\in \ll$, we calculate
 {\small\begin{eqnarray*}
 P(x_{(1)})\o P(x_{(2)})\hspace{-3mm}&\stackrel{(\ref{eq:p})}{=}&\hspace{-3mm}P(r^1)\o P([x,r^2])+P([x, r^1])\o P(r^2),\\
 P^2(x)_{(1)}\o P^2(x)_{(2)}\hspace{-3mm}&\stackrel{(\ref{eq:p})}{=}&\hspace{-3mm}r^1\o [P^2(x), r^2]+[P^2(x), r^1]\o r^2,\\
 P(P(x)_{(1)})\o P(x)_{(2)}\hspace{-3mm}&\stackrel{(\ref{eq:p})}{=}&\hspace{-3mm}P(r^1)\o [P(x), r^2]+P([P(x), r^1])\o r^2\\
 \hspace{-3mm}&=&\hspace{-3mm}P(r^1)\o [P(x), r^2]-P([r^1, P(x)])\o r^2\\
 \hspace{-3mm}&\stackrel{(\ref{eq:dualrep-1})}{=}&\hspace{-3mm}P(r^1)\o [P(x), r^2]+[N(r^1), P(x)]\o r^2-P([N(r^1), x])\o r^2-[r^1, P^2(x)]\o r^2,\\
 P(x)_{(1)}\o P(P(x)_{(2)})\hspace{-3mm}&\stackrel{(\ref{eq:p})}{=}&\hspace{-3mm}r^1\o P([P(x), r^2])+[P(x), r^1]\o P(r^2)\\
 \hspace{-3mm}&=&\hspace{-3mm}-r^1\o P([r^2, P(x)])+[P(x), r^1]\o P(r^2)\\
 \hspace{-3mm}&\stackrel{(\ref{eq:dualrep-1})}{=}&\hspace{-3mm}r^1\o [N(r^2), P(x)]-r^1\o P([N(r^2), x])-r^1\o [r^2, P^2(x)]+[P(x), r^1]\o P(r^2),
 \end{eqnarray*}
 }
 then Eq.(\ref{eq:nc}) $\Leftrightarrow$ Eq.(\ref{eq:nc-1}).

 (\ref{it:pq2}) For all $x\in \ll$, by Eq.(\ref{eq:p}), one can obtain
 {\small\begin{eqnarray*}
 P(N(x)_{(1)})\o N(x)_{(2)}\hspace{-3mm}&=&\hspace{-3mm}P(r^1)\o [N(x), r^2]+P([N(x), r^1])\o r^2\\
 \hspace{-3mm}&\stackrel{(\ref{eq:dualrep-1})}{=}&\hspace{-3mm}P(r^1)\o [N(x), r^2]+[N(x), P(r^1)]\o r^2+P([x, P(r^1)])\o r^2-[x, P^2(r^1)]\o r^2,\\
 x_{(1)}\o N^2(x_{(2)})\hspace{-3mm}&=&\hspace{-3mm}r^1\o N^2([x, r^2])+[x ,r^1]\o N^2(r^2),\\
 P(x_{(1)})\o N(x_{(2)})\hspace{-3mm}&=&\hspace{-3mm}P(r^1)\o N([x, r^2])+P([x, r^1])\o N(r^2),\\
 N(x)_{(1)}\o N(N(x)_{(2)})\hspace{-3mm}&=&\hspace{-3mm}r^1\o N([N(x), r^2])+[N(x), r^1]\o N(r^2)\\
 \hspace{-3mm}&\stackrel{(\ref{eq:na})}{=}&\hspace{-3mm}r^1\o [N(x), N(r^2)]+r^1\o N^2([x, r^2])-r^1\o N([x, N(r^2)])+[N(x), r^1]\o N(r^2),
 \end{eqnarray*}
 }
 then Eq.(\ref{eq:nliebialg2}) $\Leftrightarrow$ Eq.(\ref{eq:nliebialg2-1}).
 \end{proof}

 \begin{lem} \label{lem:thm:pq}
 \begin{enumerate}
   \item If $(P\o \id-\id\o N)(r)=0$, then $(ad_x\circ P^2\o \id-ad_x\o N^2)(r)=0$.
   \item If $r$ is antisymmetric and $(P\o \id-\id\o N)(r)=0$, then $(N\o \id-\id\o P)(r)=0$.
 \end{enumerate}
 \end{lem}

 \begin{proof} Straightforward.
 \end{proof}

 By Theorem \ref{thm:pq} and Lemma \ref{lem:thm:pq}, one gets

 \begin{thm}\label{thm:nqt} Let $((\ll, [,]), N)$ be a $P$-admissible Nijenhuis Lie algebra and $\d:=\d_r$ given by Eq.(\ref{eq:p}). If $r$ is an antisymmetric solution of the cYBe in $(\ll, [,])$ and the equation below holds
 \begin{eqnarray}
 &(P\o \id-\id\o N)(r)=0,& \label{eq:cybe-1}
 \end{eqnarray}
 then $((\ll, [,]), \d_r, N, P)$ is a \N Lie bialgebra. In this case, we call $((\ll, [,]), \d_r, N, P)$ {\bf triangular}.
 \end{thm}

 \begin{ex}\label{ex:qtnliebialg} In fact, the \N Lie bialgebras given in Example \ref{ex:nliebialg} are all triangular, with $r=\l (f\o g-g\o f)$.
 \end{ex}

 \begin{defi} \label{de:cybe-1} Let $((\ll, [,]), N)$ be a $P$-admissible Nijenhuis Lie algebra. Then the cYBe (i.e. Eq.(\ref{eq:cybe})) together with Eq.(\ref{eq:cybe-1}) is call a {\bf classical $P$-Nijenhuis Yang-Baxter equation} (cPNYBe for short).
 \end{defi}

 Let $V$ be a vector space and $r\in V\o V$, define $\rr: V^*\lr V$ by
 \begin{eqnarray}\label{eq:rr}
 &\rr(v^*)=\langle v^*, r^1 \rangle r^2, ~~\forall v^*\in V^*.&
 \end{eqnarray}
 The element $r$ is {\bf nondegenerate} if $\rr$ is bijective.

 \begin{thm} \label{thm:rr} Let $((\ll, [,]), N)$ be a Nijenhuis algebra, $r\in \ll\o \ll$ antisymmetric and $P:\ll\lr \ll$ a linear map. Then $r$ is a solution of the cPNYBe if and only if for all $a^*, b^*\in \ll^*$, $\rr$ satisfies
 \begin{eqnarray}\label{eq:rr-1}
 &[\rr(a^*), \rr(b^*)]=\rr(ad^*(\rr(a^*))(b^*)-ad^*(\rr(b^*))(a^*))&
 \end{eqnarray}
 and
 \begin{eqnarray}\label{eq:rr-2}
 N(\rr(a^*))=\rr(P^*(a^*)).
 \end{eqnarray}
 \end{thm}

 \begin{proof} By \cite{Maj}, we know that $r$ is a solution of the cYBe if and only if Eq.(\ref{eq:rr-1}) holds. And for all $a^*\in \ll^*$, one gets
 \begin{eqnarray*}
  &\rr(P^*(a^*))=\langle P^*(a^*), r^1 \rangle r^2=\langle a^*, P(r^1) \rangle r^2,&\\
  &P(\rr(a^*))=\langle a^*, r^1 \rangle P(r^2).&
 \end{eqnarray*}
 So Eq.(\ref{eq:rr-2}) $\Leftrightarrow$ Eq.(\ref{eq:cybe-1}). We finish the proof.
 \end{proof}

 \begin{defi}\label{de:rbo} Let $((\ll, [,]), N)$ be a Nijenhuis Lie algebra, $(V, \rho)$ a \rep of $(\ll, [,])$ and $\a: V\lr V$ a linear map. A linear map $T: V\lr  \ll$ is called a {\bf weak relative Rota-Baxter operator (rRBO for short) associated to $(V, \rho)$ and $\a$} if for all $u, v\in V$, $T$ satisfies
 \begin{eqnarray}\label{eq:rbo-1}
 &[T(u), T(v)]=T(\rho(T(u))(v)-\rho(T(v))(u)),&
 \end{eqnarray}
 and
 \begin{eqnarray}\label{eq:rbo-2}
 &N(T(u))=T(\a(u)).&
 \end{eqnarray}
 Furthermore, if $((V, \rho), \a)$ is a representation of $(((\ll, [,]), N), N)$, then $T$ is called a {\bf rRBO associated to $((V, \rho), \a)$.}
 \end{defi}

 By Theorem \ref{thm:rr}, one can obtain

 \begin{cor}\label{cor:rbo}  Let $((\ll, [,]), N)$ be a Nijenhuis algebra, $r\in \ll\o \ll$ antisymmetric and $P:\ll\lr \ll$ a linear map. Then $r$ is a solution of the cPNYBe if and only if $\rr$ is a weak rRBO associated to $(\ll^*, ad^*)$ and $P^*$. If furthermore, $((\ll, [,]), N)$ is $P$-admissible, then $r$ is a solution of the cPNYBe if and only if $\rr$ is a rRBO associated to $((\ll^*, ad^*), P^*)$.
 \end{cor}

 Triangular Nijenhuis Lie bialgebras can be derived from seimi-product Lie algebras by using rRBOs.

 \begin{lem}\label{lem:rbosemi} Let $((\ll, [,]), N)$ be a Nijenhuis Lie algebra, $(V, \rho)$ a \rep of $(\ll, [,])$, $P:\ll\lr \ll$ and $\a, \b:V\lr V$ linear maps. Then TFAE:
 \begin{enumerate}
   \item\label{it:rbosemi1} $(\ll\ltimes_{\rho} V, N+\a)$ is a $(P+\b)$-admissible Nijenhuis Lie algebra.
   \item \label{it:rbosemi2} $(\ll\ltimes_{\rho^*} V^*, N+\b^*)$ is a $(P+\a^*)$-admissible Nijenhuis Lie algebra.
   \item \label{it:rbosemi3} The following are satisfied:
    \begin{enumerate}[(1)]
      \item \label{it:rbosemi3-1} $((V, \rho), \a)$ is a representation of $((\ll, [,]), N)$;
      \item \label{it:rbosemi3-2} $((\ll, [,]), N)$ is $P$-admissible;
      \item \label{it:rbosemi3-3} $((\ll, [,]), N)$ is $\b$-admissible associated to $(V, \rho)$;
      \item \label{it:admdual3-4} For $x\in \ll, v\in V$, we have
        \begin{eqnarray}\label{eq:rbosemi3-1}
        &\b(\rho(x)(\a(v)))+\rho(P^2(x))(v)=\rho(P(x))(\a(v))+\b(\rho(P(x))(v)).&
        \end{eqnarray}
    \end{enumerate}
 \end{enumerate}
 \end{lem}

 \begin{proof} ((\ref{it:rbosemi1}) $\Leftrightarrow$ (\ref{it:rbosemi3})) By Proposition \ref{pro:repnlie}, $(\ll\ltimes_{\rho}V, N+\a)$ is a Nijenhuis Lie algebra if and only if $((V, \rho), \a)$ is a representation of $((\ll, [,]), N)$. Eq.(\ref{eq:dualrep-1}) holds for $[,]_\heartsuit, N+\a, P+\b$ if and only if, for $x, y\in \ll$, $u, v\in V$,
 \begin{eqnarray*}
 &&P([N(x), y])+\b(\rho(N(x))(v))-\b(\rho(y)(\a(u)))+[x, P^2(y)]+\rho(x)\b^2(v)-\rho(P^2(y))u\\
 &&=[N(x), P(y)]+\rho(N(x))(\b(v))-\rho(P(y))(\a(u))+P([x, P(y)])+\b(\rho(x)(\b(v)))-\b(\rho(P(y))(u))
 \end{eqnarray*}
 if and only if $((\ll, [,]), N)$ is $P$-admissible, $((\ll, [,]), N)$ is $\b$-admissible associated to $(V, \rho)$ and Eq.(\mref{eq:rbosemi3-1}) holds.

 ((\ref{it:rbosemi2}) $\Leftrightarrow$ (\ref{it:rbosemi3})) Based on (\ref{it:rbosemi1}) $\Leftrightarrow$ (\ref{it:rbosemi3}), (\ref{it:rbosemi2}) is equivalent to
    \begin{enumerate}[(i)]
      \item \label{it:rbosemi3-1-1} $((V^*, \rho^*), \b^*)$ is a representation of $((\ll, [,]), N^*)$, i.e., $((\ll, [,]), N)$ is $\b$-admissible associated to $(V, \rho)$;
      \item \label{it:rbosemi3-2-1} $((\ll, [,]), N)$ is $P$-admissible;
      \item \label{it:rbosemi3-3-1} $((\ll, [,]), N)$ is $\a^*$-admissible associated to $(V^*, \rho^*)$, i.e., $((V, \rho), \a)$ is a representation of $((\ll, [,]), N)$;
      \item \label{it:admdual3-4-1} For $x\in \ll, a^*\in V^*$, we have
        \begin{eqnarray*}\label{eq:rbosemi3-1-1}
        &\a^*(\rho(x)(\b^*(a^*)))+\rho(P^2(x))(a^*)=\rho(P(x))(\b^*(a^*))+\a^*(\rho(P(x))(a^*)),&
        \end{eqnarray*}
        i.e., Eq.(\ref{eq:rbosemi3-1}) holds.
    \end{enumerate}
    Thus the proof of Lemma \ref{lem:rbosemi} is finished.
 \end{proof}

 The following result provides a method to construct \N Lie bialgebras on semi-direct product \N Lie algebras.
 \begin{thm} \label{thm:rboliebialg} Let $((\ll, [,]), N)$ be a $\b$-admissible \N Lie algebra associated to $(V, \rho)$, $P:\ll\lr \ll$, $\a: V\lr V$ be linear maps. Assume that $T: V\lr \ll$ is a linear map, which is identified as an element $\sum_{i=1}^n T(e_i)\o e^i\in (\ll\ltimes_{\rho^*} V^*)\o (\ll\ltimes_{\rho^*} V^*)$, where $\{e_i\}_{i=1}^n$ is a basis of $V$ and $\{e^i\}_{i=1}^n$ is the dual basis.
 \begin{enumerate}
   \item \label{it:rboliebialg-a} The element $r=\sum_{i=1}^n (T(e_i)\o e^i-e^i\o T(e_i))$ is an antisymmetric solution of the classical $(P+\a^*)$-\N Yang-Baxter equation in the Nijenhuis Lie algebra $(\ll\ltimes_{\rho^*} V^*, N+\b^*)$ if and only if $T$ is a weak rRBO associated to $(V, \rho)$ and $\a$, and $T \circ \b=P\circ T$.
   \item \label{it:rboliebialg-b} Assume that $((V, \rho), \a)$ is a representation of $((\ll, [,]), N)$. If $T$ is a rRBO associated to $((V, \rho), \a)$ and $T\circ \b=P\circ T$, then $r=\sum_{i=1}^n (T(e_i)\o e^i-e^i\o T(e_i))$ is an antisymmetric solution of the classical $(P+\a^*)$-\N Yang-Baxter equation in the Nijenhuis Lie algebra $(\ll\ltimes_{\rho^*} V^*, N+\b^*)$. If further, $((\ll, [,]), N)$ is $P$-admissible and Eq.(\ref{eq:rbosemi3-1}) holds, then $(\ll\ltimes_{\rho^*}V^*, N+\b^*)$ is $(P+\a^*)$-admissible. In this case, there is a Nijenhuis Lie bialgebra $((\ll\ltimes_{\rho^*} V^*, N+\b^*), \d_r, P+\a^*)$, where $\d_r$ is given by Eq.(\ref{eq:p}) with $r=\sum_{i=1}^n (T(e_i)\o e^i-e^i\o T(e_i))$.
 \end{enumerate}
 \end{thm}

 \begin{proof} (\ref{it:rboliebialg-a}) By \cite{Bai2}, $r=\sum_{i=1}^n (T(e_i)\o e^i-e^i\o T(e_i))$ is an antisymmetric solution of the Eq.(\ref{eq:cybe}) in the Lie algebra $\ll\ltimes_{\rho^*} V^*$ if and only if $T$ satisfies Eq.(\ref{eq:rbo-1}) for $(V, \rho)$. It remains to prove that Eq.(\ref{eq:cybe-1}) for $P+\a^*$ and $N+\b^*$ $\Leftrightarrow$ $T \circ \b=P\circ T$ and Eq.(\ref{eq:rbo-2}) hold. In fact,
 \begin{eqnarray*}
 ((P+\a^*)\o \id-\id\o (N+\b^*))(r)\hspace{-3mm}&=&\hspace{-4mm}\sum_{i=1}^n(P(T(e_i))\o e^i-\a^*(e^i)\o T(e_i)-T(e_i)\o \b^*(e^i)+e^i\o N(T(e_i)))\\
 &=&\hspace{-4mm}\sum_{i=1}^n(P(T(e_i))\o e^i-e^i\o T(\a(e_i))-T(\b(e_i))\o e^i+e^i\o N(T(e_i))).
 \end{eqnarray*}
 Therefore we can finish the proof of (\ref{it:rboliebialg-a}).

 \noindent
 (\ref{it:rboliebialg-b}) can be checked by Item (\ref{it:rboliebialg-a}) and Lemma \ref{lem:rbosemi}.
 \end{proof}

\section{Nijenhuis operators on $\ss$} \label{se:nsl2} In this section, we give all the Nijenhuis operators on $\ss$ with respect to the standard Cartan-Weyl basis.

\begin{thm}\label{thm:nsl2} Let $\{e, f, h\}$ be the standard Cartan-Weyl basis and the map $N:\ss\lr \ss$ given by
    \begin{center}$N(e, f, h):=(e,f,h)\left(
                    \begin{array}{ccc}
                      k_1 & k_4 & k_7 \\
                      k_2 & k_5 & k_8 \\
                      k_3 & k_6 & k_9 \\
                    \end{array}
                  \right),$
    \end{center}
    where $k_i, i=1, 2, \cdots, 9$ are parameters. Then all the Nijenhuis operators $N$ on $(\ss, [,])$ are given as follows.
    \begin{center}$(1)~\left(
                    \begin{array}{ccc}
                      k_1 & k_4 & 0 \\
                      0 & k_5 & -4 k_3\\
                      k_3 &k_6 & k_1 \\
                    \end{array}
                  \right),$ \quad\smallskip
                  $(2)~\left(
                    \begin{array}{ccc}
                      k_1 & 0 & 0 \\
                      0 & k_5 & 0\\
                      k_3 &\frac{(k_9-k_1)(k_5-k_9)}{4 k_3} & k_9 \\
                    \end{array}
                  \right)~(k_3\neq 0),$  \quad \smallskip
                  $(3)~\left(
                    \begin{array}{ccc}
                      k_1 & & \\
                      k_2 & k_5 & k_8\\
                      \frac{k_2(k_9-k_5)}{ k_8}-\frac{k_8}{4} &\frac{k_8(k_1-k_9)}{4 k_2}  & k_9 \\
                    \end{array}
                  \right)~(k_2,k_8\neq 0),$\quad\smallskip
                  $(4)\left(
                    \begin{array}{ccc}
                      k_1 & 0 & -4 k_6\\
                      k_2 & k_5 & 0 \\
                      k_3& k_6  & k_5 \\
                    \end{array}
                  \right),$ \quad \smallskip
                  $(5)\left(
                    \begin{array}{ccc}
                      k_1 & k_4 & k_7 \\
                      \frac{k_3 k_7}{k_4}  & k_5 &-4 k_3\\
                      k_3& \frac{k_4(k_5-k_1)}{k_7}-\frac{k_7}{4} & k_5 \\
                    \end{array}
                  \right)~(k_4, k_7\neq 0),$ \quad\smallskip
                  $(6)~\left(
                    \begin{array}{ccc}
                      k_1 & 0 & 0 \\
                      0 & k_5 & 0\\
                      k_3 &k_6 & k_5 \\
                    \end{array}
                  \right)~(k_5\neq k_1),$ \quad \smallskip
                  $(7)~\left(
                    \begin{array}{ccc}
                      k_1 & k_4 & k_7 \\
                      k_2  & k_9+\frac{4 k_3 k_4}{k_7}-\frac{4 k_2 k{_4}^2}{k{_7}^2} & -\frac{4 k_2 k_4}{k_7}\\
                      k_3 & \frac{k_4(k_9-k_1)}{k_7}-\frac{k_7}{4}  & k_9 \\
                    \end{array}
                  \right)~(k_4, k_7\neq 0, k_2 \neq \frac{k_3 k_7}{k_4}).$
    \end{center}
    \end{thm}

    \begin{proof} It can be obtained by a direct computation.
    \end{proof}

 \begin{cor}\label{cor:thm:nsl2} The diagonalisable Nijenhuis operator on $(\ss, [,])$ under the standard Cartan-Weyl basis $\{e, f, h\}$ must be
 \begin{center}$\left(
                    \begin{array}{ccc}
                      k_1 & 0 & 0 \\
                      0 & k_5 & 0 \\
                      0 & 0 & k_5 \\
                    \end{array}
                  \right)$\quad or\quad
                  $\left(
                    \begin{array}{ccc}
                      k_1 & 0 & 0 \\
                      0 & k_5 & 0 \\
                      0 & 0 & k_1 \\
                    \end{array}
                  \right),$
    \end{center}
    where $k_1, k_5$ are parameters.
 \end{cor}

 In Theorem \ref{thm:nsl2}, we give all the Nijenhuis operators on $(\ss, [,])$. Now we discuss the relation between Nijenhuis operators in Example \ref{ex:nsl2} and in Theorem \ref{thm:nsl2}. 
 \begin{rmk}\label{rmk:nsl2}
  \begin{enumerate}[(1)]
    \item \label{it:rmk:nsl2-1} Comparison between Nijenhuis operators in Example \mref{ex:nsl2} and Nijenhuis operators in Theorem \ref{thm:nsl2}.
 \begin{center}
  \begin{tabular}[t]{c|c}
         \hline\hline
        \bf{Nij operators in Example \mref{ex:nsl2}}  & \bf{Nij operators in Theorem \ref{thm:nsl2}}\\
        \hline
         $\left\{
            \begin{array}{l}
        N(e)=-\l \kappa f,\\
        N(f)=2 \l \kappa f, \\
        N(h)=2 \l \kappa h,\\
            \end{array}
            \right.$
        & $k_1=k_3=k_6=0, k_2=-\l \kappa, k_5=2\l \kappa$ in Case (4)\\
        \hline
    \end{tabular}
    \end{center}\smallskip
  \item \label{it:rmk:nsl2-2} Based on Item \ref{it:rmk:nsl2-1}, the Nijenhuis operators on $\ss$ obtained by the methods in Theorem \ref{thm:n} are only a part of them given in Theorem \ref{thm:nsl2}, therefore the conditions in Theorem \ref{thm:n} are sufficient and not necessary.
  \end{enumerate}
 \end{rmk}

 \section*{Acknowledgment} The authors are deeply indebted to the referee for his/her very useful suggestions and some improvements to the original manuscript. Ma is supported by National Natural Science Foundation of China (No. 12471033). Li is supported by National Natural Science Foundation of China (No. 12471130).

 \smallskip

 \noindent
 {\bf Declaration of interests.} The authors have no conflicts of interest to disclose.

 \noindent
 {\bf Data availability.} Data sharing is not applicable as no new data were created or analyzed.

\end{document}